\documentclass[a4paper,12pt]{article}
\usepackage[utf8]{inputenc} 
\usepackage[T1]{fontenc}   
\usepackage{hyperref}     
\usepackage{url}           
\usepackage{booktabs}      
\usepackage{amsfonts}       
\usepackage{nicefrac}      
\usepackage{microtype}     
\usepackage{lipsum}
\usepackage{amssymb}
\usepackage{amsmath}
\usepackage{textcomp}
\usepackage{eurosym}
\usepackage{verbatim}
\usepackage{amsthm}
\usepackage{commath}
\usepackage{mathtools}
\usepackage{graphicx}
\usepackage{caption}
\usepackage{subcaption}
\usepackage{dsfont}
\usepackage{xcolor}
\usepackage[numbers]{natbib}
\usepackage{float}
\usepackage{wrapfig}
\usepackage{mathtools}
\usepackage{bigints}

\theoremstyle{plain}
\newtheorem{thm}{Theorem}[section]

\newtheorem{coro}[thm]{Corollary}

\theoremstyle{definition}
\newtheorem{ex}[thm]{Example}
\newtheorem{rem}[thm]{Remark}

\newcommand{\N}{\mathbb{N}}
\newcommand{\R}{\mathbb{R}}

\def\E{\mathbb{E}}
\def\d{\, \mathrm{d}}
\def\1{\mathds{1}}
\def\R{\mathbb{R}}

\def\E{\mathbb{E}}
\def\N{\mathbb{N}}

\def\|{\, | \,}

\def\rd{\,\mathrm{d}}

\def\var{\mathrm{Var}}
\def\cov{\mathrm{Cov}}

\def\nn{\nonumber}

\def\eqas{\overset{a.s.}{=}}
\numberwithin{equation}{section}
\allowdisplaybreaks

\begin{document}

\title{The Quadratic Variation of Gauss-Markov Semimartingales}
\author{Georges Kassis  \\ \\
Department of Mathematics, University College London \\ London WC1E 6BT, United Kingdom}
\date{October 28, 2024}
\maketitle

\begin{abstract}
The covariance function of a Gauss-Markov process evaluated at points $(s,t)$ admits a representation as a product of a function of $\min(s,t)$ and a function of $\max(s,t)$. We call these functions the covariance factors of a Gauss-Markov process, and give the expression of the quadratic variation of a Gauss-Markov semimartingale in terms of its covariance factors.

\vspace{.25cm}

\end{abstract}

{\bf Keywords}: Gauss-Markov process, quadratic variation, covariance

\vspace{.25cm}

\section{Introduction}
Let $I$ be a real set of the form $[a,b)$ where $-\infty < a < b \leqslant \infty$. The covariance function $K_X$ of a Gauss-Markov process $(X_t)_{t\in I}$ satisfies the following relationship: 
\begin{equation} \label{relrel}
K_X(r,s)K_X(s,t)= K_X(s,s)K_X(r,t)
\end{equation}
for all $(r,s,t)\in I^3$ such that $r\leqslant s \leqslant t$, see \cite{Feller}, Chapter 3, Section 8. Note that this result is originally from a 1945 paper in French by Loève "Fonctions aléatoires du second ordre" \cite{Loeve}, which was later included in Lévy's book "Processus stochastiques et mouvement Brownien" \cite{Levy} in 1948 as a supplementary chapter. The result was later translated and included in Loève's book "Probability Theory" \cite{Loeve2} in 1963. The relationship (\ref{relrel}) is necessary and sufficient for a Gaussian process to be Markovian. Furthermore, it follows from the Markov property that, whenever $\var[X_s]=0$ for $s\in int(I)$, we have that $K_X(r,t)=0$ for all $(r,t) \in I^2$ such that $r\leqslant s \leqslant t$. We denote by $\mathcal R_X$ the boundary of the set of roots of the variance of $(X_t)$ strictly larger than $a$ and strictly smaller than $\sup(I)$, in increasing order.  Hence, if $(X_t)$ is continuous in probability, that is, its mean and covariance functions are continuous on $I$ and $I^2$ respectively, the ordered isolated roots in $\mathcal R_X$ define what we call the \textit{independence planes} of $(X_t)$: $K_X(s,t)$ is guaranteed to be zero unless $s$ and $t$ are in the same independence plane, i.e., the biggest element in $\mathcal R_X$ smaller than $s$ is also the biggest element in $\mathcal R_X$ smaller than $t$ and the smallest element in $\mathcal R_X$ bigger than $s$ is also the smallest element in $\mathcal R_X$ bigger than $t$. A direct consequence of that observation is that $K_X$ can be understood as a stitching of several covariance functions, each defined on one of the independence planes:
\begin{equation} \label{decom}
    K_X(s,t)= \sum_{i=0}^{|\mathcal R_X|} K_i(s,t) \1_{[z_i,z_{{i+1}})}(s,t),
\end{equation}
where $z_0:=a$, $z_1<z_2<\ldots$ are the elements of $\mathcal R_X$, $|\mathcal R_X| \in \N \cup \{\infty\}$ is the cardinal of $\mathcal R_X$, and $K_i$ is a covariance function for each $i=0,\ldots, |\mathcal R_X|$, i.e., a real symmetric positive definite function, on $[z_i,z_{{i+1}})$. When $|\mathcal R_X|$ is finite, we set $z_{|\mathcal R_X|+1}= \sup(I)$.

If $(X_t)$ is not continuous in probability, $z_1,z_2,\ldots$ do not characterise all the independence planes of $(X_t)$: it could be that there exists a point $s\in I$ such that $\var[X_s] \neq 0$ and $K_X(r,t)=0$ for all $(r,t) \in I^2$ such that $r < s < t$. Adding all such points $s$ to $\mathcal R_X$ while making sure that if one of these points $s$ is such that $K_X(s,t)=K_X(r,s)=0\neq \var[X_s]$ for all $(r,t) \in I^2$ such that $r < s < t$, it is added twice, and ordering the obtained set yields the \textit{independence partition} $\mathcal I_X=(\iota_1\leqslant\iota_2\leqslant\ldots)$ of the Gauss-Markov process $(X_t)$. Then, letting $\N_{|\mathcal I_X|}:= \N \cap \{0,\ldots,|\mathcal I_X|\}$, $\iota_0:=a$, and, when $|\mathcal I_X|$ is finite, $\iota_{|\mathcal I_X|+1}:= \sup(I)$, Eq. $(\ref{decom})$ can be adapted for stochastically discontinuous processes:
\begin{equation} \label{decom2}
    K_X(s,t)= \sum_{i=0}^{|\mathcal I_X|} K_i(s,t) \1_{\Gamma_i}(s,t),
\end{equation} 
where the sets $(\Gamma_i \mid i \in \N_{|\mathcal I_X|} )$ are such that 
\begin{enumerate}
    \item $(\Gamma_i=(\iota_i,\iota_{{i+1}}) \lor \Gamma_i=[\iota_i,\iota_{{i+1}}] \lor \Gamma_i=[\iota_i,\iota_{{i+1}}) \lor \Gamma_i=(\iota_i,\iota_{{i+1}}])$ holds true for each $ i \in \N_{|\mathcal I_X|}$,
    \item $\bigcup\limits_{i=0}^{|\mathcal I_X|} \Gamma_i=I$, $\bigcap\limits_{i=0}^{|\mathcal I_X|} \Gamma_i = \emptyset$,
\end{enumerate}
and $K_i:\Gamma_i \rightarrow \R$ is a covariance function (a trivial one if $\Gamma_i$ is a singleton) for each $ i \in \N_{|\mathcal I_X|}$. It follows that if $(X_t)$ is continuous in probability, Eq. $(\ref{decom2})$ is identical to Eq. $(\ref{decom})$.

Regardless of any continuity consideration, each $K_i$ must satisfy Eq. (\ref{relrel}) since restricting a Gauss-Markov process on a sub-interval yields a Gauss-Markov process. And Eq. (\ref{relrel}) implies the existence of real functions $f_i: \Gamma_i \rightarrow \R$ and $g_i:\Gamma_i \rightarrow \R$ such that 
\begin{equation}
    K_i(s,t) = f_i(\min(s,t)) g_i (\max(s,t))
\end{equation}
where $f_i/g_i$ is a positive nondecreasing function on $\Gamma_i$. To see this, let $x_i\in (\iota_i,\iota_{i+1})$ (if $(\iota_i,\iota_{i+1})$ is empty, i.e., $\Gamma_i$ is a singleton, then the result follows directly, so we exclude this case now) and choose for instance
\begin{equation}
f_i(x)= 
\begin{cases}K_i(x, x_i) & \text { for } x \leqslant x_i, \\ 
\frac{K_i(x, x) K_i(x_i, x_i) }{ K_i(x_i, x)} & \text { for } x>x_i,
\end{cases}
\end{equation}
\begin{equation}
g_i(x)= 
\begin{cases}\frac{K_i(x, x)}{K_i(x, x_i)}   & \text { for } x \leqslant x_i, \\ 
\frac{K_i(x_i, x)}{K_i(x_i, x_i)}   & \text { for } x>x_i.
\end{cases}
\end{equation}
Using Eq. (\ref{relrel}), we see that $f_i(\min(s,t)) g_i (\max(s,t)) = K_i(s,t)$. The positivity of $f_i/g_i$ follows from the positivity of the variance, and the monotonicity of $f_i/g_i$ follows from the Cauchy-Schwarz inequality applied to $K_i$. One can also show that the existence of such functions $f_i$ and $g_i$ is sufficient for Eq. (\ref{relrel}) to be satisfied. It is immediate that
\begin{equation}
     f_i(r) g_i (s)  f_i(s) g_i (t) =  f_i(s) g_i (s)  f_i(r) g_i (t)
\end{equation}
for all $(r,s,t)\in \Gamma_i^3$ such that $r\leqslant s \leqslant t$, but one must also show that $\\ f_i(\min(s,t)) g_i (\max(s,t))$ is a covariance function on $\Gamma_i$ for the sufficiency to hold. Let $k\in \N_0$, $(x_1<\ldots <x_k)\in \Gamma_i^k$, and $(y_1,\ldots, y_k)\in \R^k$. The positive definite property can be checked as follows: 
\begin{align}
    &\sum_{u,j=1}^k f_i(\min(x_u,x_j)) g_i (\max(x_u,x_j)) y_u y_j \nn \\
    &= \sum_{u=1}^k f_i(x_u) g_i(x_u) y_u^2+2 \sum_{u=1}^{k-1} \sum_{j=u+1}^k f_i(x_u) y_u  g_i(x_j) y_j  \\
    &= \sum_{u=1}^k \frac{f_i(x_u)}{g_i(x_u)} \left[g_i(x_u) y_u\right]^2  + 2  \sum_{u=1}^{k-1}\frac{f_i(x_u)}{g_i(x_u)} g_i(x_u) y_u \sum_{j=u+1}^k g_i(x_j) y_j\\
    &= \frac{f_i(x_k)}{g_i(x_k)} \left[g_i(x_k) y_k\right]^2 + \sum_{u=1}^{k-1}\frac{f_i(x_u)}{g_i(x_u)} \left( \left[g_i(x_u) y_u\right]^2 + 2 g_i(x_u) y_u \sum_{j=u+1}^k g_i(x_j) y_j \right) \\
    &= \frac{f_i(x_k)}{g_i(x_k)} \left[g_i(x_k) y_k\right]^2 + \sum_{u=1}^{k-1}\frac{f_i(x_u)}{g_i(x_u)} \left(\left[\sum_{j=u}^k g_i(x_j) y_j\right]^2 - \left[\sum_{j=u+1}^k g_i(x_j) y_j\right]^2 \right) \\
    &= \frac{f_i(x_k)}{g_i(x_k)} \left[g_i(x_k) y_k\right]^2+\sum_{u=1}^{k-1} \frac{f_i(x_u)}{g_i(x_u)}\left[\sum_{j=u}^k g_i(x_j) y_j\right]^2-\sum_{u=1}^{k-1} \frac{f_i(x_u)}{g_i(x_u)}\left[\sum_{j=u+1}^k g_i(x_j) y_j\right]^2 \\
    &= \frac{f_i(x_1)}{g_i(x_1)}\left[\sum_{j=1}^k g_i(x_j) y_j\right]^2+\sum_{u=1}^{k-1}\left[\frac{f_i(x_{u+1})}{g_i(x_{u+1})}-\frac{f_i(x_u)}{g_i(x_u)}\right]\left[\sum_{j=u+1}^k g_i(x_j) y_j\right]^2 \geqslant 0.
\end{align}
This means that a Gaussian process $(X_t)$ is Markovian if and only if there exists real functions $(f_i :\Gamma_i \rightarrow \R \mid i \in \N_{|\mathcal I_X|})$ and $(g_i :\Gamma_i \rightarrow \R \mid i \in \N_{|\mathcal I_X|})$ such that $f_i/g_i$ is positive nondecreasing on $\Gamma_i$ for all $i \in \N_{|\mathcal I_X|}$, and 
\begin{equation}
    K_X(s,t)= \sum_{i=0}^{|\mathcal I_X |} f_i(\min(s,t)) g_i (\max(s,t))  \1_{\Gamma_i}(s,t).
\end{equation}
We call any ordered sets of functions $(f_i \mid i \in \N_{|\mathcal I_X|})$ and $(g_i \mid i \in \N_{|\mathcal I_X|})$ satisfying the above equality the \textit{covariance factors} of the Gauss-Markov process $(X_t)$.
We use these factors to determine the expression of the quadratic variation of a Gauss-Markov semimartingale. Note that in the case of a Gauss-Markov semimartingale, $\Gamma_i= [\iota_i, \iota_{i+1})$ for all $i \in \N_{|\mathcal I_X|}$, and $\iota_i \neq \iota_j$ for any $(i,j)\in \N_{|\mathcal I_X|}^2$ such that $i \neq j$.
\section{Result}
Let $(X_t)_{t\in I}$ be a real Gauss-Markov semimartingale. By definition, this process is càdlàg and may have at most countably many almost sure jumps, only occurring at deterministic times. Hence, the covariance factors $(f_i,g_i)$ are at least piecewise continuous on $[\iota_i,\iota_{i+1})$, for all $i \in \N_{|\mathcal I_X|}$. This means that for each $i \in \N_{|\mathcal I_X|}$, letting $\N_{q(i)}:= \N \cap \{0, \ldots, q(i) \}$ where $q(i)$ is the sum of the number of discontinuities of $f_i$ and the number of discontinuities of $g_i$, there exist real functions $f_{i,j}$, $g_{i,j}$ defined on real intervals $A_j^i$, $B_j^i$ for all $j\in \N_{q(i)}$ such that
\begin{enumerate}
    \item $\bigcup\limits_{j=0}^{q(i)} A_j^i = \bigcup\limits_{j=0}^{q(i)} B_j^i = [\iota_i,\iota_{i+1})$, and $\bigcap\limits_{j=0}^{q(i)} A_j^i = \bigcap\limits_{j=0}^{q(i)} B_j^i = \emptyset$,
    \item $f_{i,j}:A_j \rightarrow \R$ is continuous, and $g_{i,j}:B_j \rightarrow \R$ is continuous, for all $j\in \N_{q(i)}$,
    \item  $f_i(x)= \sum\limits_{j=0}^{q(i)} f_{i,j}(x) \1_{A_j^i}(x)$, and $g_i(x)= \sum\limits_{j=0}^{q(i)} g_{i,j}(x) \1_{B_j^i}(x)$.
\end{enumerate}
We call the concatenation of the ordered sets $(f_{i,j} \mid j\in \N_{q(i)} )$ for all $i \in \N_{|\mathcal I_X|}$ and the concatenation of the ordered sets $(g_{i,j} \mid j\in \N_{q(i)} )$ for all $i \in \N_{|\mathcal I_X|}$ the \textit{continuous decompositions of the covariance factors} of the Gauss-Markov semimartingale $(X_t)$. Let 
\begin{equation}
    \mathcal D_X = \bigcup_{i=0}^{|\mathcal I_X |} \left( \{ t \in (\iota_i,\iota_{i+1}) \mid  f_i(t^-) \neq f_i(t^+) \}\cup \{ t \in (\iota_i,\iota_{i+1}) \mid  g_i(t^-) \neq g_i(t^+) \}\right)
\end{equation}
be the set of times at which discontinuities occur in the covariance factors of $(X_t)$, and let $\mathcal T_X = (\tau_1,\tau_2,\ldots)$  be the ordered set of elements in $\mathcal I_X \cup \mathcal D_X$ in increasing order. We introduce the integral notation $\int_\cdot^{t^-}$ for $\lim_{x\rightarrow t^-} \int_\cdot^{x}$ in the next result.
\begin{thm}
Let $(X_t)_{t \in I}$ be a real centred Gauss-Markov semimartingale. Then, for $t \in I$,
\begin{equation} \label{quadvar}
    [X]_t\eqas \int_{\tau_{m(t)}}^{t^-} g_{m(t)}^2(s) \rd \left(\frac{f_{m(t)}}{g_{m(t)}}(s) \right) + (\Delta X_{t})^2 +  \sum_{i=0}^{m(t)-1} \int_{\tau_i}^{\tau_{i+1}^-} g_i^2(s) \rd \left(\frac{f_i}{g_i}(s) \right)  + (\Delta X_{\tau_{i+1}})^2,
\end{equation}
where
\begin{enumerate}
    \item $\tau_0:=a$ and $(\tau_1, \tau_2, \ldots )$ is the ordered set $\mathcal T_X$,
    \item  the ordered sets $(f_0, f_1, \ldots )$ and $(g_0, g_1, \ldots )$ are the continuous decompositions of the covariance factors of $(X_t)$,
    \item $m(t):= \max \{i \in \N \mid \tau_i < t \}$, and $\Delta X_s :=  X_s - \lim\limits_{u\rightarrow s^-} X_u$.
\end{enumerate}
\end{thm}
We give a few remarks before giving the proof.
\begin{rem} \phantom{.}
\begin{enumerate}
    \item All the Riemann-Stieltjes integrals (see Section 36 in \cite{Kolmogorov}) that appear in Eq. (\ref{quadvar}) are well defined since the integrators $s \mapsto f_i(s)/g_i(s)$, $i=0,1, \ldots$, are monotone functions on the domains of integration $[\tau_i, \tau_{i+1})$, $i=0,1, \ldots$, of their respective integrals. 
    \item If the process $(X_t)$ is sample-continuous, then $\Delta X_s = 0$ for all $s\in I$, which means that the quadratic variation $[X]_t$ is a deterministic function. 
    \item The sum $\sum_{\tau_0 \leqslant s \leqslant t} (\Delta X_s)^2$ contains at most countably many terms. This is because $(X_t)$ is a semimartingale, hence a càdlàg process. 
    \item The process $(X_t)$ might not jump at every time $\tau_i$ for $i=1,2,\ldots$. It could also not jump at all. This is because $\mathcal T_X$ contains all jump times if there are any and might contain more than just jump times. For instance, if we stitch two independent Brownian bridges together, the time at which the stitching happens is in $\mathcal T_X$, but the process does not jump anywhere.
\end{enumerate}
\end{rem}
\begin{proof}
Since $[X]_t= [X]_{[\tau_{m(t)},t]} + \sum_{i=0}^{m(t)-1} [X]_{[\tau_i,\tau_{i+1}]}$, it is enough to prove that
\begin{equation}
    [X]_{\tau_1}\eqas  \int_{\tau_0}^{\tau_1^-} g_0^2(s) \d \left(\frac{f_0}{g_0}(s) \right) + (\Delta X_{\tau_1})^2.
\end{equation}
Let $(P_n)_{n\in \N_0}$ be a sequence of partitions of $[\tau_0,\tau_1]$ such that $P_n:=\{t_0, t_1, \ldots, t_n \| \tau_0=t_0< t_1< \ldots< t_{n-1} < t_n=\tau_1\}$, and $\Delta_j= X_{t_{j+1}} - X_{t_j}$. Then,
\begin{align}
    \E \left[ \sum_{j=0}^{n-1} \Delta_j^2 \right] &= \sum_{j=0}^{n-2} \E [X_{t_{j+1}}^2] + \E[X_{t_{j}}^2] - 2\E [X_{t_{j+1}}X_{t_{j}}] \nn\\ 
    & \hspace{2cm} +\E [X_{{\tau_1}}^2] + \E[X_{t_{n-1}}^2] - 2\E [X_{\tau_1}X_{t_{n-1}}] \nn\\ 
    &=\sum_{j=0}^{n-2} f_0(t_{j+1})g_0(t_{j+1}) + f_0(t_{j})g_0(t_{j}) - 2 f_0(t_{j})g_0(t_{j+1}) \nn \\ 
    & \hspace{2cm} +f_1(\tau_1)g_1(\tau_1) + f_0(t_{n-1})g_0(t_{n-1}) - 2 f_0(t_{n-1})g_1(\tau_1) \nn\\
    &=\sum_{j=0}^{n-2} g_0(t_{j+1})(f_0(t_{j+1}) - f_0(t_{j})) - f_0(t_{j})(g_0(t_{j+1}) - g_0(t_{j})) \nn \\
    & \hspace{2cm} +g_1(\tau_1)(f_1(\tau_1) - f_0(t_{n-1})) - f_0(t_{n-1})(g_1(\tau_1) - g_0(t_{n-1})) \nn\\
    &=\sum_{j=0}^{n-2} g_0(t_{j})g_0(t_{j+1}) \frac{g_0(t_{j+1})(f_0(t_{j+1}) - f_0(t_{j})) - f_0(t_{j})(g_0(t_{j+1}) - g_0(t_{j}))}{g_0(t_{j})g_0(t_{j+1})} \nn \\
    & \hspace{2cm} +g_1(\tau_1)(f_1(\tau_1) - f_0(t_{n-1})) - f_0(t_{n-1})(g_1(\tau_1) - g_0(t_{n-1}))\nn \\
    &=: S_{n-2} + s_{n-1}.
\end{align}
When $n$ increases as the mesh of the partition goes to $0$, we get 
\begin{equation}
    \lim\limits_{n\rightarrow \infty} S_{n-2}  = \int_{\tau_0}^{\tau_1^-} g_0^2(s) \d \left(\frac{f_0}{g_0}(s) \right),
\end{equation}
and
\begin{align}
    \lim\limits_{n\rightarrow \infty} s_{n-1} &= g_1(\tau_1)(f_1(\tau_1) - f_0(\tau_1^-)) - f_0(\tau_1^-)(g_1(\tau_1) - g_0(\tau_1^-)) \nn\\
    &=f_1({\tau_1})g_1({\tau_1}) + f_0({\tau_1^-})g_0({\tau_1^-}) - 2 f_0({\tau_1^-})g_1({\tau_1}) \nn\\
    &=\E [X_{\tau_1}^2] + \E[X_{{\tau_1^-}}^2] - 2\E [X_{{\tau_1}}X_{{\tau_1^-}}] \nn\\
    &= \E[(\Delta X_{\tau_1})^2].
\end{align}
This means that
\begin{equation}
    \lim\limits_{n\rightarrow \infty} \E[ [X]_{\tau_1}^{P_n} - (\Delta X_{\tau_1})^2] = \int_{\tau_0}^{\tau_1^-} g_0^2(s) \d \left(\frac{f_0}{g_0}(s) \right).
\end{equation}
Furthermore,
\begin{align}
    &\var \left[ \sum_{j=0}^{n-1} \Delta_j^2 - (\Delta X_{\tau_1})^2 \right] \nn \\
    &=\E\left[\left( \left(\sum_{j=0}^{n-1} \Delta_j^2\right) - (\Delta X_{\tau_1})^2\right)^2\right] - \E^2\left[ \left(\sum_{j=0}^{n-1} \Delta_j^2\right) - (\Delta X_{\tau_1})^2\right]\\
    &= \E \left[\left(\sum_{j=0}^{n-1} \Delta_j^2\right)^2 \right] + \E[ (\Delta X_{\tau_1})^4 ] -2 \sum_{j=0}^{n-1}\E\left[\Delta_j^2(\Delta X_{\tau_1})^2 \right] - \E^2\left[ \left(\sum_{j=0}^{n-1} \Delta_j^2\right) - (\Delta X_{\tau_1})^2\right] \\
    &= 2 \sum_{j=0}^{n-1} \sum_{i=0}^{n-1} \E^2 \left[\Delta_i \Delta_j \right]+  \sum_{j=0}^{n-1} \sum_{i=0}^{n-1}  \E \left[\Delta_i^2 \right]\E \left[ \Delta_j^2 \right]+ 3\E^2[ (\Delta X_{\tau_1})^2 ] \nn \\
    &\hspace{1cm} -4 \sum_{j=0}^{n-1}\E^2\left[\Delta_j(\Delta X_{\tau_1}) \right]-2 \sum_{j=0}^{n-1}\E\left[\Delta_j^2(\Delta X_{\tau_1})^2 \right] - \E^2\left[ \left(\sum_{j=0}^{n-1} \Delta_j^2\right) - (\Delta X_{\tau_1})^2\right],
\end{align}
where we used in the last step the fact that for two centred and jointly Gaussian random variables $Y$ and $Z$, we have $\cov(Y^2,Z^2)=2\cov^2(Y,Z)$.
To conclude the proof, we need to show that this variance goes to zero as $n$ goes to infinity. We start by looking at the first term $2 \sum\limits_{j=0}^{n-1} \sum\limits_{i=0}^{n-1} \E^2 \left[\Delta_i \Delta_j \right]$:
\begin{equation}
\E \left[\Delta_i \Delta_j \right] =  \left\{
    \begin{array}{ll}
        (f_0(t_{i+1}) - f_0(t_{i}))(g_0(t_{j+1}) - g_0(t_{j}))  & \mbox{for } i<j<n-1,\\ \\
         (f_0(t_{j+1}) - f_0(t_{j}))(g_0(t_{i+1}) - g_0(t_{i}))  & \mbox{for } j<i<n-1, \\ \\
         g_0(t_{j+1})(f_0(t_{j+1}) - f_0(t_{j})) - f_0(t_{j})(g_0(t_{j+1}) - g_0(t_{j})) & \mbox{for } i=j<n-1, \\ \\
        (f_0(t_{i+1}) - f_0(t_{i}))(g_1({\tau_1}) - g_0(t_{n-1}))  & \mbox{for } i<j=n-1,\\ \\
         (f_0(t_{j+1}) - f_0(t_{j}))(g_1({\tau_1}) - g_0(t_{n-1}))  & \mbox{for } j<i=n-1, \\ \\
         g_1({\tau_1})(f_1({\tau_1}) - f_0(t_{n-1})) - f_0(t_{n-1})(g_1({\tau_1}) - g_0(t_{n-1})) & \mbox{for } i=j=n-1.
    \end{array}
\right.
\end{equation}
We split the sum into the cases $i<j$, $j<i$, and $i=j$. For $i=j$,
\begin{equation}
    2 \sum_{j=0}^{n-2}  \E^2 \left[\Delta_j^2 \right] \leqslant 2\max_{k \in\{0,1,\ldots, n-2 \}} \E [ \Delta_k^2] \sum_{j=0}^{n-2}  \E [ \Delta_j^2] \xrightarrow[n \rightarrow \infty]{} 0, 
\end{equation}
and so
\begin{equation}
    2 \sum_{j=0}^{n-1}  \E^2 \left[\Delta_j^2 \right] = 2 \sum_{j=0}^{n-2}  \E^2 \left[\Delta_j^2 \right] + 2\E^2 \left[\Delta_{n-1}^2 \right]  \xrightarrow[n \rightarrow \infty]{} 2\E^2[ (\Delta X_{\tau_1})^2 ].
\end{equation}
For $i<j$, we have 
\begin{align}
    2 \sum_{j=0}^{n-1} \sum_{i=0}^{j-1} \E^2 \left[\Delta_i \Delta_j \right]&= 2 \sum_{i=0}^{n-2} \E^2 \left[\Delta_i \Delta_{n-1} \right]+ 2 \sum_{j=0}^{n-1} \sum_{i=0}^{j-1} \E^2 \left[\Delta_i \Delta_j \right]  \nn \\
    &= 2\sum_{i=0}^{n-2}(f_0(t_{i+1}) - f_0(t_{i}))^2(g_1({\tau_1}) - g_0(t_{n-1}))^2  \nn \\
    & \hspace{2cm} 2 \sum_{j=0}^{n-2} \sum_{i=0}^{j-1} (f_0(t_{i+1}) - f_0(t_{i}))^2(g_0(t_{j+1}) - g_0(t_{j}))^2 \nn \\ 
    &=2 (g_1({\tau_1}) - g_0(t_{n-1}))^2\sum_{i=0}^{n-2} (f_0(t_{i+1}) - f_0(t_{i}))^2  \nn \\
    & \hspace{2cm} + 2\sum_{j=0}^{n-2} (g_0(t_{j+1}) - g_0(t_{j}))^2 \left( \sum_{i=0}^{j-1} (f_0(t_{i+1}) - f_0(t_{i}))^2\right)\nn \\ 
    &\leqslant 2\left((g_1({\tau_1}) - g_0(t_{n-1}))^2+ \sum_{j=0}^{n-2} (g_0(t_{j+1}) - g_0(t_{j}))^2\right)  \nn \\
    & \hspace{2cm} \left( \sum_{i=0}^{n-2}(f_0(t_{i+1}) - f_0(t_{i}))^2\right) \nn \\
    & \xrightarrow[n \rightarrow \infty]{} 0
\end{align}
since $(g_1({\tau_1}) - g_0(t_{n-1}))^2+ \sum_{j=0}^{n-2} (g_0(t_{j+1}) - g_0(t_{j}))^2 \xrightarrow[n \rightarrow \infty]{} (g_1(\tau_1)-g_0(\tau^-))^2$ and $\sum\limits_{i=0}^{n-2} (f_0(t_{i+1}) - f_0(t_{i}))^2 \xrightarrow[n \rightarrow \infty]{} [f_0]_{\tau_1^-}=0$. The same argument can be applied to the case $j<i$ to show that 
\begin{equation} 
     2 \sum_{i=0}^{n-1} \sum_{j=0}^{i-1} \E^2 \left[\Delta_i \Delta_j \right] \xrightarrow[n \rightarrow \infty]{} 0.
\end{equation}
Combining the three cases, we conclude that 
\begin{equation} \label{eq1}
    2 \sum_{j=0}^{n-1} \sum_{i=0}^{n-1} \E^2 \left[\Delta_i \Delta_j \right] \xrightarrow[n \rightarrow \infty]{} 2\E^2[ (\Delta X_{\tau_1})^2 ].
\end{equation}
We already know that the second term appearing in the variance converges: 
\begin{equation}
    \sum_{j=0}^{n-1} \sum_{i=0}^{n-1}  \E \left[\Delta_i^2 \right]\E \left[ \Delta_j^2 \right] =\left(\sum_{i=0}^{n-1}  \E \left[\Delta_i^2 \right]\right)^2 \xrightarrow[n \rightarrow \infty]{} \left(\int_{\tau_0}^{\tau_1^-} g_0^2(s) \d \left(\frac{f_0}{g_0}(s) \right) +\E[(\Delta X_{\tau_1})^2] \right)^2,
\end{equation}
and that the third term does not depend on $n$. Now we consider the fourth term \\ $-4 \sum_{j=0}^{n-1}\E^2\left[\Delta_j(\Delta X_{\tau_1}) \right]$: 
\begin{align}
    -4 \sum\limits_{j=0}^{n-1}\E^2\left[\Delta_j(\Delta X_{\tau_1}) \right] &= -4 \sum\limits_{j=0}^{n-2}\E^2\left[\Delta_j(\Delta X_{\tau_1}) \right] -4 \E^2\left[\Delta_{n-1}(\Delta X_{\tau_1})\right] \nn \\
    &=-4(g_1(\tau) - g_0(\tau^{-}))^2\sum_{j=0}^{n-2} (f_0(t_{j+1}) - f_0(t_{j}))^2  -4 \E^2\left[\Delta_{n-1}(\Delta X_{\tau_1})\right] \nn \\
    & \xrightarrow[n \rightarrow \infty]{} -4(g_1(\tau) - g_0(\tau^{-}))^2 [f_0]_{\tau_1^-} - 4 \E^2\left[(\Delta X_{\tau_1})^2\right] = -4 \E^2\left[(\Delta X_{\tau_1})^2\right]. \label{eq2}
\end{align}
We already know the limit of the fifth and sixth terms appearing in the variance:
\begin{align}
    -2 \sum_{j=0}^{n-1}\E\left[\Delta_j^2(\Delta X_{\tau_1})^2 \right]&= -2\E\left[(\Delta X_{\tau_1})^2 \right] \sum_{j=0}^{n-1}\E\left[\Delta_j^2 \right]  \nn \\
    &\xrightarrow[n \rightarrow \infty]{}  -2\E\left[(\Delta X_{\tau_1})^2 \right] \int_{\tau_0}^{\tau_1^-} g_0^2(s) \d \left(\frac{f_0}{g_0}(s) \right), \\
    - \E^2\left[ \left(\sum_{j=0}^{n-1} \Delta_j^2\right) - (\Delta X_{\tau_1})^2\right] &\xrightarrow[n \rightarrow \infty]{} -\left( \int_{\tau_0}^{\tau_1^-} g_0^2(s) \d \left(\frac{f_0}{g_0}(s) \right) \right)^2.
\end{align}
Combining all the limits, we conclude that
\begin{align}
    \var \left[ [X]^{P_n}_{\tau_1} - (\Delta X_{\tau_1})^2 \right]  \xrightarrow[n \rightarrow \infty]{} &2\E^2[ (\Delta X_{\tau_1})^2 ] + \left(\int_{\tau_0}^{\tau_1^-} g_0^2(s) \d \left(\frac{f_0}{g_0}(s) \right) +\E[(\Delta X_{\tau_1})^2] \right)^2 \nn \\
    &\quad + 3\E^2[ (\Delta X_{\tau_1})^2 ] -4 \E^2\left[(\Delta X_{\tau_1})^2\right] \nn \\
    &\quad -2\E\left[(\Delta X_{\tau_1})^2 \right] \int_{\tau_0}^{\tau_1^-} g_0^2(s) \d \left(\frac{f_0}{g_0}(s) \right) \nn \\
    &\quad -\left( \int_{\tau_0}^{\tau_1^-} g_0^2(s) \d \left(\frac{f_0}{g_0}(s) \right) \right)^2 =0.
\end{align}
Hence,
\begin{equation}
    [X]_{\tau_1}\eqas  \int_{\tau_0}^{\tau_1^-} g_0^2(s) \d \left(\frac{f_0}{g_0}(s) \right) + (\Delta X_{\tau_1})^2.
\end{equation}
\end{proof}
We can now treat the case where the mean of $(X_t)$ is not zero.
\begin{coro}
Let $(Y_t)_{t \in I}$ be a real Gauss-Markov semimartingale, $\mu_Y(t):= \E[Y_t]$, and $X_t:= Y_t - \mu_Y(t)$. Then, for $t \in I$,
\begin{equation}
    [Y]_t= [X]_t + 2 \sum_{a \leqslant s \leqslant t} \Delta X_s \Delta \mu_Y(s) + \sum_{a \leqslant s \leqslant t}   (\Delta \mu_Y(s))^2.
\end{equation}
\end{coro}
\begin{proof}
First, we show that $\mu_Y$ is of bounded variation. Let $t \in I$. Let $(h_n)_{n \in \N_0}$ be a sequence of real left-continuous step functions on $[0,t]$ decreasing to $0$. We define the elementary integrals $V_n:= \int_{a}^t h_n(s) \d Y_s$, and $v_n:=\E[V_n]=\int_{a}^t h_n(s) \d \mu_Y(s)$. Since $(Y_t)$ is a semimartingale, the sequence of Gaussian random variables $(V_n)$ converges to $0$ in probability. This implies that it also converges to $0$ in $L^1$ (see Section 13.7 in \cite{Williams}), which means that the sequence of real numbers $(v_n)$ converges to $0$. Hence $\mu_Y$ is of bounded variation. Returning to the quadratic variation of $(Y_t)$, we have 
\begin{align}
    [Y]_t&= [X + \mu_Y]_t \nn \\
    &= [X]_t  + 2 [X, \mu_Y]_t + [\mu_Y]_t.
\end{align}
Since $\mu_Y$ is of bounded variation, we have $[\mu_Y]_t = \sum_{a \leqslant s \leqslant t} (\Delta\mu_Y(s))^2$ and 
\begin{equation}
    [X, \mu_Y]_t = \int_a^t \Delta X_s \d \mu_Y(s)= \sum_{a \leqslant s \leqslant t} \Delta X_s \Delta \mu_Y(s),
\end{equation}
see Theorem 26.6 in \cite{Kallenberg}.
\end{proof}
We illustrate the result with a few examples.
\begin{ex}
If $I=[0,\infty)$ and $(X_t)_{t\in I}$ is a standard Brownian motion, then $\mathcal T_X= \emptyset$, and the covariance factors are $f_0(x)= x$ and $g_0(x)=1$. Hence, $[X]_t=\int_0^{t^-} \d s = t^- = t$ for any $t\in I$. More generally, if $(X_t)_{t\in I}$ is a continuous Gauss-Markov semimartingale with independent increments, then $\mathcal T_X= \emptyset$, and the covariance factors are a positive non-decreasing function $f_0(x)$ and $g_0(x)=1$. Hence, $[X]_t=\int_0^{t} \d f_0(x)$ for any $t\in I$.
\end{ex}
\begin{ex}
If $I=[0,\infty)$ and $(X_t)_{t\in I}$ is a continuous stationary\footnote{A stochastic process $(X_t)_{t\in I}$ is said to be stationary if it is square integrable, has constant mean, and $\cov(X_s,X_t)$ only depends on $\abs{s-t}$. } Gauss-Markov semimartingale, then $\mathcal T_X= \emptyset$, and the covariance factors are $f_0(x)= \alpha \exp[\beta x]$ and $g_0(x)=\exp[-\beta x]$ for $(\alpha,\beta)\in \R^2_+$ (this follows from the derivation of the covariance function of an Ornstein-Uhlenbeck process in \cite{Ornstein}). Hence, $[X]_t=\alpha \int_0^{t} \exp[-2 \beta x] \rd (\exp[2 \beta x]) = 2 \alpha\beta \int_0^{t}  \rd x= 2 \alpha\beta t$ for any $t\in I$.
\end{ex}
\begin{ex}
If $I=[0,\infty)$ and $(X_t)_{t\in I}$ is a Gauss-Markov semimartingale defined by $\mu_X(t)=0$ and 
\begin{equation}
K_X(s,t) =  \left\{
    \begin{array}{ll}
        \min(s,t)(1-\max(s,t))  & \mbox{for } s,t < 1,\\
        \min(s,t) -1  & \mbox{for } 1 \leqslant s,t < 2 \mbox{ or } s,t \geqslant 2, \\
        0 & \mbox{otherwise}.
    \end{array}
\right.
\end{equation}
then $\mathcal T_X= (1,2)$, and the covariance factors are $f_0(x)= x$, $f_1(x)= f_2(x)=x-1$, $g_0(x)=1-x$, $g_1(x)=g_2(x)=1$, which are all continuous. Note that at $t=1$, the semimartingale is continuous, but not at $t=2$. Hence, 
\begin{equation}
[X]_t  =  \left\{
    \begin{array}{ll}
        \int_0^{t^-} (1-x)^2 \d (x/(1-x)) =t  & \mbox{for } t < 1,\\
        1 + \int_1^{t^-} \d x = t & \mbox{for } 1 \leqslant t < 2, \\
        2+ \int_2^{t^-} \d x + \Delta X_{2}= t + \Delta X_{2}  & \mbox{for } t \geqslant 2,
    \end{array}
\right.
\end{equation}
where $\Delta X_{2} \sim \mathcal N(0,2)$.
\end{ex}
\begin{ex}
If $I=[0,1)$ and $(X_t)_{t\in I}$ is a Gauss-Markov semimartingale defined by $\mu_X(t)=\1_{[0,1/4)} (t) - \1_{[1/4,1/2)} (t)$ and 
\begin{equation}
K_X(s,t) =  \left\{
    \begin{array}{ll}
        \min(s,t)  & \mbox{for } s < 1/2,\\
         \min(s,t) +1  & \mbox{for } s > 1/2 \mbox{ and  } t>1/2,
    \end{array}
\right.
\end{equation}
then $\mathcal T_X= (1/2)$, and the covariance factors are $f_0(x)= x \1_{[0,1/2)} (x) + (x +1)\1_{[1/2,1)} (x)$ and $g_0(x)=1$. Since $f_0$ is not continuous, we must use the continuous decompositions of the covariance factors, i.e., $(x,x+1)$ and $(1,1)$. Hence, 
\begin{equation}
[X-\mu_X]_t  =  \left\{
    \begin{array}{ll}
        t  & \mbox{for } t < 1/2,\\
        t + \Delta(X-\mu_X)_{1/2}  & \mbox{for } t \geqslant 1/2,
    \end{array}
\right.
\end{equation}
where $\Delta(X-\mu_X)_{1/2} \sim \mathcal N(0,1)$, which implies
\begin{equation}
[X]_t  =  \left\{
    \begin{array}{ll}
        t  & \mbox{for } t < 1/4,\\
        t + 4  & \mbox{for } 1/4 \leqslant t < 1/2, \\
        t + 3\Delta(X-\mu_X)_{1/2}+5 & \mbox{for } t \geqslant 1/2.
    \end{array}
\right.
\end{equation}
\end{ex}
\section*{Acknowledgments}
G. Kassis acknowledges the UCL Department of Mathematics for a Teaching Assistantship Award. Furthermore, the author is grateful to Andrea Macrina for useful discussions and feedback.


\begin{thebibliography}{99}
\bibitem{Feller} Feller, W. (1967) \textit{Introduction to Probability Theory and Its Applications, Vol 2.} New York: John Wiley.
\bibitem{Kallenberg} Kallenberg, O. (2002) \textit{Foundations of Modern Probability.} New York: Springer-Verlag.
\bibitem{Kolmogorov} Kolmogorov, A. N. and Fomin, S. V. (1975) \textit{Introductory Real Analysis.} Translated from the Russian by R. A. Silverman. New York: Dover Publications.
\bibitem{Levy} Lévy, P. (1948) \textit{Processus stochastiques et mouvement Brownien}. Paris: Gauthier-Villars.
\bibitem{Loeve} Loève, M. (1945) \textit{Fonctions aléatoires du second ordre}. Revue Scientifique, 83, pp. 297–303.
\bibitem{Loeve2} Loève, M. (1963) \textit{Probability Theory}. 3rd ed. Princeton, N.J.: Van Nostrand.
\bibitem{Ornstein} Ornstein, L.S. and Uhlenbeck, G.E. (1930) \textit{On the theory of the Brownian motion.}. Physical Review, 36(5), pp.823–841.
\bibitem{Williams} Williams, D. (1991) \textit{Probability with martingales.} Cambridge: Cambridge University Press. 
\end{thebibliography}
\end{document}